\documentclass{ws-ijgmmp}

\usepackage{graphicx}
\usepackage[breaklinks=true,colorlinks=true,linkcolor=blue,urlcolor=blue,citecolor=blue]{hyperref}

\usepackage{bbold}

\usepackage{psfrag}

\usepackage{color}

\begin{document}

\markboth{Mikhail Panine, Achim Kempf}
{A Convexity Result in the Spectral Geometry of Conformally Equivalent Metrics on Surfaces}

%
\catchline{}{}{}{}{}
%

\title{A Convexity Result in the Spectral Geometry of Conformally Equivalent Metrics on Surfaces}

\author{Mikhail Panine}

\address{Department of Applied Mathematics, University of Waterloo,\\ 200 University Ave West, Waterloo, Ontario, N2L 3G1, Canada\\
\email{mpanine@uwaterloo.ca}}

\author{Achim Kempf}

\address{Department of Applied Mathematics, University of Waterloo,\\ 200 University Ave West, Waterloo, Ontario, N2L 3G1, Canada\\
\email{akempf@uwaterloo.ca} }

\maketitle

\begin{history}
\received{(Day Month Year)}
\revised{(Day Month Year)}
\end{history}




\begin{abstract}
Motivated by considerations of euclidean quantum gravity, we investigate a central question of spectral geometry, namely the question of reconstructability of compact Riemannian manifolds from the spectra of their Laplace operators. To this end, we study analytic paths of metrics that induce isospectral Laplace-Beltrami operators over oriented compact surfaces without boundary. Applying perturbation theory, we show that sets of conformally equivalent metrics on such surfaces contain no nontrivial convex subsets. This indicates that cases where the manifolds cannot be reconstructed from their spectra are highly constrained. 

\end{abstract}


\maketitle

\section{Introduction}

In the field of quantum gravity, an important challenge is to separate the physically relevant degrees of freedom in the (pseudo-)Riemannian metric from spurious degrees of freedom that merely express diffeomorphism invariance \cite{gibbons1993euclidean,rovelli2004quantum,kiefer2007quantum,landi1997general}. This makes geometric invariants of (pseudo-) Riemannian manifolds, such as the eigenvalues of Laplace or d'Alembert operators of particular interest. 
In euclidean quantum gravity, therefore, a central question of spectral geometry arises, namely the question of to what extent the geometry of compact Riemannian manifolds is determined by the spectra of their Laplace or Dirac operators. The  application of spectral geometric methods in quantum gravity is promising also because spectral geometry naturally combines the language of general relativity, differential geometry, with the language of quantum theory, functional analysis \cite{kempf2010spacetime,aasen2013shape,QGQC,panine2016towards}.

It is known that under certain conditions, reconstruction of the `shapes' of compact Riemannian manifolds from the geometric invariants given by spectra of their Laplace operators is possible \cite{panine2016towards,zelditch2009inverse,hezari2010inverse,zelditch1998revolution}. See also \cite{datchev2011inverse} for a review of such results. However, it is also known that the reconstruction can not always be unique. For reviews, see \cite{gordon2000survey,giraud2010hearing}. 
The question arises, therefore, whether counterexamples to the reconstructability of shapes from spectra, that is the existence of isospectral non-isometric manifolds, are in some sense rare.

In this paper, we give new evidence for the rarity of isospectral non-isometric manifolds by studying the properties of sets of isospectral metrics over a fixed differentiable manifold. Topological properties of isospectral sets of Riemannian metrics, mainly compactness, have attracted some interest over the years, see \cite{osgood1988compact,brooks1989isospectral,brooks1992compactness,chang1990isospectral} among others. Here, we  study the types of analytic paths that can occur within isospectral sets of conformally equivalent metrics. This is done via perturbation theoretic techniques analogous to those used in quantum mechanics. Consequently, we begin by discussing the main technique that we will here use, in the context of  quantum mechanical perturbation theory. Section \ref{BabyExample} is dedicated to this task. Then, in Section \ref{ChangingProduct} we derive the formulas for the first two orders of the eigenvalue corrections for a perturbation of both the studied operator and the inner product on the studied Hilbert space. Finally, in Section \ref{IsospectralFamilies} we apply the above to conformal perturbations of metrics on surfaces. This allows us to show the main result of this paper, namely that sets of conformally equivalent Riemannian metrics on surfaces with isospectral Laplace-Beltrami operators contain no convex subsets. The same holds for sets of inverse (``upper indices'') metrics.

\section{Strategy} \label{BabyExample}

In this section, we use standard (quantum mechanical) perturbation theory as an illustrative example for the rest of our study. We use this example both to set the notation and explain the strategy used to prove the main results of this paper. In what follows, we use $t$ as the perturbation parameter and we denote orders of perturbation of various objects by superscripts in parentheses. We assume the perturbation series to converge in some neighborhood of $t=0$. See \cite{rellich1969perturbation,kato1966perturbation,reed1978iv} for the relevant convergence criteria.

Let $H^{(0)}$ be a self-adjoint operator on a separable Hilbert space $\mathcal{H}$ with inner product $\langle \cdot , \cdot \rangle$ and let $\{\psi^{(0)}_{n},\lambda^{(0)}_{n}\}_{n}$ be its orthonormal eigenvectors and eigenvalues:

\begin{equation}
    H \psi^{(0)}_{n} = \lambda^{(0)}_{n}\psi^{(0)}_{n}
\end{equation}

\noindent We assume that $\mathcal{H}$  is spanned by the $\{\psi^{(0)}_{n}\}_{n}$. We perturb $H^{(0)}$ by adding a first order correction $tH^{(1)}$, assumed self-adjoint. 

\begin{equation}
\begin{aligned}
&H(t) = H^{(0)}+tH^{(1)} \\
&\psi_{n} (t) = \psi^{(0)}_{n}+t\psi^{(1)}_{n}+t^{2}\psi_{n}^{(2)}+...\\
&\lambda_{n} (t) = \lambda^{(0)}_{n}+t\lambda^{(1)}_{n}+t^{2}\lambda_{n}^{(2)}+...
\end{aligned}
\end{equation}

\noindent Even though, for now, only a first order correction to the operator is assumed, eigenvectors and eigenvalues are corrected to all orders. Supposing the spectrum of $H^{(0)}$ nondegenerate, standard perturbation theory provides the following formulas for the first eigenvalue corrections \cite{tannoudji1973mecanique,liboff2003introductory}:

\begin{equation} \label{UsualPerturbation}
    \begin{aligned}
        &\lambda_{n}^{(1)} = \langle \psi^{(0)}_{n}, H^{(1)}\psi^{(0)}_{n}\rangle \\
        &\lambda_{n}^{(2)} = \sum_{i \neq n} \frac{ \left| \langle \psi^{(0)}_{i}, H^{(1)}\psi^{(0)}_{n}\rangle \right|^{2}}{\lambda_{n}^{(0)}-\lambda_{i}^{(0)}}
    \end{aligned}
\end{equation}

\noindent Suppose that $H(t)$ is an isospectral family. Consequently, all corrections to the eigenvalues must vanish. In particular, in the first two orders,

\begin{equation}
    \begin{aligned}
        &0 = \langle \psi^{(0)}_{n}, H^{(1)}\psi^{(0)}_{n}\rangle \\
        &0 = \sum_{i \neq n} \frac{ \left| \langle \psi^{(0)}_{i}, H^{(1)}\psi^{(0)}_{n}\rangle \right|^{2}}{\lambda_{n}^{(0)}-\lambda_{i}^{(0)}}
    \end{aligned}
\end{equation}

\noindent The first order equation tells us that the diagonal of $H^{(1)}$ in the basis $\{\psi_{n}^{(0)}\}_{n}$ vanishes. The second equation restricts the possible values of the non-diagonal elements $\langle \psi^{(0)}_{i}, H^{(1)}\psi^{(0)}_{n}\rangle$ of $H^{(1)}$ in the basis $\{\psi_{n}^{(0)}\}_{n}$. This indicates that, under suitable additional hypotheses, one could force $H^{(1)}$ to vanish entirely.

Before proceeding any further with the analysis of these expressions, a short conceptual digression is in order. Notice that here perturbation theory is applied in a way opposite to the usual one.
Indeed, instead of using our knowledge of $H^{(1)}$ and the $\{\psi_{n}^{(0)} \}_{n}$ to compute the eigenvalue corrections, the vanishing of those corrections is used to deduce some properties of $H^{(1)}$.
No knowledge of the $\{\psi_{n}^{(0)} \}_{n}$ is required beyond the properties guaranteed by the spectral theorem.

Returning to the analysis of the perturbative expressions, consider the case where the spectrum of the unperturbed operator $H^{(0)}$ can be numbered such that $\lambda^{(0)}_k > \lambda^{(0)}_j$ if $k > j$ with the lowest eigenvalue denoted by $\lambda^{(0)}_{0}$. In other words, the operator is bounded below, this bound is attained and the spectrum cannot have certain types of accumulation points. Then, the elements $\langle \psi^{(0)}_{i}, H^{(1)}\psi^{(0)}_{n}\rangle$ are forced to vanish. Indeed, the second order correction to the lowest eigenvalue can be expressed as:

\begin{equation}
    0 = -\sum_{i > 0} \frac{ \left| \langle \psi^{(0)}_{i}, H^{(1)}\psi^{(0)}_{0}\rangle \right|^{2}}{\left|\lambda_{0}^{(0)}-\lambda_{i}^{(0)}\right|}
\end{equation}

\noindent This implies that $\left| \langle \psi^{(0)}_{i}, H^{(1)}\psi^{(0)}_{0}\rangle \right|^{2}$ vanishes for all $i > 0$. Since $H^{(1)}$ is self-adjoint, $\left| \langle \psi^{(0)}_{i}, H^{(1)}\psi^{(0)}_{0}\rangle \right|^{2} = \left| \langle \psi^{(0)}_{0}, H^{(1)}\psi^{(0)}_{i}\rangle \right|^{2} =0 $ for all $i>0$. Thus, for $n=1$ the second order correction becomes,

\begin{equation}
\begin{aligned}
    0 &= \frac{ \left| \langle \psi^{(0)}_{0}, H^{(1)}\psi^{(0)}_{1}\rangle \right|^{2}}{\lambda_{1}^{(0)}-\lambda_{0}^{(0)}} -\sum_{i > 1} \frac{ \left| \langle \psi^{(0)}_{i}, H^{(1)}\psi^{(0)}_{1}\rangle \right|^{2}}{\left|\lambda_{1}^{(0)}-\lambda_{i}^{(0)}\right|}\\
    &=-\sum_{i > 1} \frac{ \left| \langle \psi^{(0)}_{i}, H^{(1)}\psi^{(0)}_{1}\rangle \right|^{2}}{\left|\lambda_{1}^{(0)}-\lambda_{i}^{(0)}\right|}
\end{aligned}
\end{equation}

\noindent In turn, this implies that $\left| \langle \psi^{(0)}_{i}, H^{(1)}\psi^{(0)}_{1}\rangle \right|^{2} =0$ for $i >1$. Repeating the above process inductively, we obtain

\begin{equation}
    \langle \psi^{(0)}_{i}, H^{(1)}\psi^{(0)}_{j}\rangle =0 ~~ \forall i,j
\end{equation}

\noindent In other words, the perturbation is trivial: $H^{(1)}=0$. This is the type of result that we endeavour to obtain in the case of the perturbation of the metric of a $2-$dimensional Riemannian manifold. A few technical difficulties arise along the way, but the main inductive strategy remains unchanged.

Before proceeding further with our geometric goals, the same inductive process can be used to rule out certain isospectral second order corrections.
Consider a perturbation $H(t) = H^{(0)}+tH^{(1)}+t^{2}H^{(2)}+...$ with $H^{(i)}$ self-adjoint. The first order correction to the eigenvalues remains unchanged, while the second order correction becomes

\begin{equation} \label{SecondOrderQMPetrurbation}
        \lambda_{n}^{(2)} = \sum_{i \neq n} \frac{ \left| \langle \psi^{(0)}_{i}, H^{(1)}\psi^{(0)}_{n}\rangle \right|^{2}}{\lambda_{n}^{(0)}-\lambda_{i}^{(0)}} + \langle \psi_{n}^{(0)} , H^{(2)}\psi_{n}^{(0)}\rangle
\end{equation}

\noindent Notice that the above inductive strategy applies if $\langle \psi_{n}^{(0)}, H^{(2)} \psi_{n}^{(0)} \rangle \leq 0$ for all $n$, yielding $H^{(1)}=0$ and $\langle \psi_{n}^{(0)}, H^{(2)} \psi_{n}^{(0)} \rangle =0$ for all $n$.

So far, our discussion was limited to nondegenerate spectra, which is the generic case \cite{bando1983generic}. 
We will now extend the above to perturbations of spectra with finite degeneracy (which is guaranteed for Laplacians).
Specifically, we endeavor to preserve the form of the perturbative expansions for $\lambda_{n}^{(1)}$ and $\lambda_{n}^{(2)}$ in order to leave the inductive argument unchanged.
In general, the computation of perturbative corrections to the eigenvalues involves a series of diagonalizations of certain operators on nested subspaces.
This is known as the reduction process, the theory of which can be found in  \cite{kato1966perturbation,baumgartel1985analytic}.
For our purposes, this corresponds to a particular choice of basis within each eigenspace of $H^{(0)}$, as is explained below.

Let $P_{n}^{(0)}$ denote the projection onto the eigenspace of $\lambda_{n}^{(0)}$.
The first order corrections $\{\lambda_{ni}^{(1)} \}_{i}$ to $\lambda_{n}^{(0)}$ are the eigenvalues of the operator $\Lambda_{n}^{(1)} = P_{n}^{(0)}H^{(1)}P_{n}^{(0)}$ restricted to the eigenspace $P_{n}^{(0)} \mathcal{H}$. Notice the additional index $i$. It enumerates the families of eigenvalues $\lambda(t)$ of $H(t)$ that are equal at $t=0$ and stay degenerate to first order.

The second order corrections to the eigenvalues are obtained in a similar manner. Let $P_{ni}^{(1)}$ denote the projection onto the eigenspace of $\Lambda_{n}^{(1)}$ corresponding to the eigenvalue correction $\lambda_{ni}^{(1)}$. The second order corrections to the $\lambda_{n}^{(0)}+t \lambda_{ni}^{(1)}$ family of eigenvalues are given by the eigenvalues $\{ \lambda_{nij}^{(2)} \}_{j}$ of the following operator on $P_{ni}^{(1)} \mathcal{H}$.

\begin{equation}
\Lambda_{ni}^{(2)} = \sum_{k \neq n} \frac{P_{ni}^{(1)} H^{(1)} P_{k}^{(0)} H^{(1)}  P_{ni}^{(1)} }{\lambda_{n}^{(0)} - \lambda_{k}^{(0)} } + P_{ni}^{(1)} H^{(2)} P_{ni}^{(1)}
\end{equation}

Once again, the new index $j$ in $\lambda_{nij}^{(2)}$ keeps track of the eigenvalues that stay degenerate to second order. Similarly, let $P_{nij}^{(2)}$ denote the projection onto the eigenspace of $\Lambda_{ni}^{(2)}$ corresponding to the eigenvalue $\lambda_{nij}^{(2)}$.
 
Notice that both $\Lambda_{n}^{(1)}$ and $\Lambda_{ni}^{(2)}$ are symmetric operators on finite dimensional vector spaces $P_{n}^{(0)} \mathcal{H}$ and $P_{ni}^{(1)} \mathcal{H}$, respectively. Consequently, their eigenvectors on those spaces can be chosen to form an orthonormal basis.
More precisely, due to the nesting of relevant eigenspaces, the choice of orthonormal eigenvectors of the $\{ \Lambda_{ni}^{(2)} \}_{i}$ fixes the choice of eigenbasis of $\Lambda_{n}^{(1)}$, which, in turn, fixes the choice of eigenbasis of $H^{(0)}$. 
Assume that the $\{\psi^{(0)}_{n}\}_{n}$ form such a basis. Then, the eigenvalue corrections become

\begin{equation}
\begin{aligned}
\lambda_{ni}^{(1)} &= \langle \psi^{(0)}_{n}, \Lambda_{n}^{(1)} \psi^{(0)}_{n}\rangle ~~~~~~~\text{for}~\psi^{(0)}_{n} \in P_{nij}^{(2)}\mathcal{H} \subseteq P_{ni}^{(1)}\mathcal{H} \\
\lambda_{nij}^{(2)} &= \langle \psi^{(0)}_{n}, \Lambda_{ni}^{(2)} \psi^{(0)}_{n}\rangle ~~~~~~~\text{for}~\psi^{(0)}_{n} \in P_{nij}^{(2)}\mathcal{H}
\end{aligned}
\end{equation}

Consequently, the additional indices $i$ and $j$ can be dropped and the expressions for the eigenvalue corrections take a form similar to the nondegenerate case.

\begin{equation} \label{FullyGeneralEigenvalues}
\begin{aligned}
&\lambda_{n}^{(1)}  = \langle \psi_{n}^{(0)} , H^{(1)} \psi_{n}^{(0)} \rangle\\
&\lambda_{n}^{(2)} = \sum_{\substack{i\\ \lambda^{(0)}_{i} \neq \lambda_{n}^{(0)} } } \frac{ \left| \langle \psi^{(0)}_{i}, H^{(1)}\psi^{(0)}_{n}\rangle \right|^{2}}{\lambda_{n}^{(0)}-\lambda_{i}^{(0)}} + \langle \psi_{n}^{(0)} , H^{(2)}\psi_{n}^{(0)}\rangle
\end{aligned}
\end{equation}

Note that the presence of higher order corrections will not spoil the above formula, as those corrections will only induce an additional choice of eigenbasis within $P_{nij}^{(2)}\mathcal{H}$. In other words, a basis $\{\psi^{(0)}_{n}\}_{n}$ in which Equation \eqref{FullyGeneralEigenvalues} holds always exists.

The form of the eigenvalue corrections given by Equation \eqref{FullyGeneralEigenvalues} is nearly identical to that for the nondegenerate case. The only difference is the omission of the terms that would contain a division by zero from the second order correction to the eigenvalues. Notice that our choice of eigenbasis guarantees that the numerator $\left| \langle \psi^{(0)}_{i}, H^{(1)}\psi^{(0)}_{n}\rangle \right|^{2}$ of those omitted terms would be zero. Thus, the vanishing of those terms can be seen without invoking the inductive argument above. Consequently, our strategy can be applied in the same way as in the nondegenerate case.


\section{Perturbations with Changing Inner Product} \label{ChangingProduct}

As our goal is to study the perturbation of the metric on a Riemannian manifold $(\mathcal{M},g)$, we are forced to also consider the change in the $\mathcal{L}_{2}(\mathcal{M})$ inner product induced by the change of the metric. Similarly, we are forced to allow for the perturbations of the studied operator not to be symmetric with respect to the unperturbed inner product. This second consideration is much more important to our endeavours, as it modifies the formulas for the eigenvalue corrections. The change in inner product, on the other hand, only influences changes in eigenvectors. Indeed, the eigenvalues of an operator on a fixed linear space do not depend upon the inner product defined on it. The norm of a vector, however, does. Thus, by requiring the perturbed eigenvectors to be normalized with respect to the perturbed inner product, we introduce a new dependence into the perturbation formulas. This dependence can be seen in the formula for the first order correction to the eigenvectors, which we provide as an intermediary step in the calculation of second order corrections to the eigenvalues (Equation \eqref{FirstOrderPsi}).

Our analysis makes no reference to the geometric setting that motivates our study and thus remains general. We assume all relevant power series to converge in some neighborhood of $t=0$. Our procedure follows the one used to derive the standard perturbation results, see \cite{tannoudji1973mecanique,liboff2003introductory,messiah1999quantum}.

Let $\mathcal{H}$ be a Hilbert space with unperturbed inner product $\langle \cdot , \cdot \rangle$. Suppose that $H^{(0)}$, $H^{(1)}$, $H^{(2)}$, $G^{(1)}$ and $G^{(2)}$ are operators on $\mathcal{H}$ such that $H^{(0)}$ is self-adjoint with respect to the unperturbed inner product and has nondegenerate spectrum. The operators $G^{(1)}$ and $G^{(2)}$ are used to perturb the inner product. For our purposes, only the first two orders of perturbation matter. Let $\langle \cdot , \cdot \rangle_{t}$ be a family of inner products on $\mathcal{H}$ defined as:

\begin{equation}
    \langle \cdot , \cdot \rangle_{t} = \langle \cdot , \cdot \rangle+ t\langle \cdot , G^{(1)}\cdot \rangle+t^{2}\langle \cdot , G^{(2)}\cdot \rangle
\end{equation}

\noindent The perturbed operator, eigenvectors and eigenvalues are given by

\begin{equation}
\begin{aligned}
&H(t) = H^{(0)}+tH^{(1)}+t^{2}H^{(1)} \\
&\psi_{n} (t) = \psi^{(0)}_{n}+t\psi^{(1)}_{n}+t^{2}\psi_{n}^{(2)}+...\\
&\lambda_{n} (t) = \lambda^{(0)}_{n}+t\lambda^{(1)}_{n}+t^{2}\lambda_{n}^{(2)}+...
\end{aligned}
\end{equation}

\noindent Order by order, up to second order, the eigenvalue equation becomes

\begin{equation} \label{PerturbedEigenvalueProblem}
\begin{split}
    H^{(0)}\psi^{(0)}_{n} &= \lambda_{n}^{(0)} \psi_{n}^{(0)}\\
    H^{(0)}\psi^{(1)}_{n} + H^{(1)}\psi^{(0)}_{n} &= \lambda_{n}^{(0)} \psi_{n}^{(1)}+\lambda_{n}^{(1)} \psi_{n}^{(0)}\\
    H^{(0)}\psi^{(2)}_{n} + H^{(1)}\psi^{(1)}_{n} + H^{(2)}\psi^{(0)}_{n} &= \lambda_{n}^{(0)} \psi_{n}^{(2)}+\lambda_{n}^{(1)} \psi_{n}^{(1)} +\lambda_{n}^{(2)} \psi_{n}^{(0)}
\end{split}
\end{equation}

\noindent Requiring that the eigenstates are normalized for all $t$ in the convergence radius of the series, i.e. that $\langle \psi_{n}(t), \psi_{n}(t)\rangle_{t} = 1$ yields, after splitting order by order

\begin{equation} \label{NormalizationCondition}
\begin{split}
    \langle \psi^{(0)}_{n}, \psi^{(0)}_{n} \rangle & =1 \\
    \langle \psi^{(0)}_{n}, \psi^{(1)}_{n} \rangle & = -\frac{1}{2} \langle \psi^{(0)}_{n}, G^{(1)}\psi^{(0)}_{n} \rangle \\
    \langle \psi^{(0)}_{n}, \psi^{(2)}_{n} \rangle & = -\frac{1}{2} \left( \langle \psi^{(1)}_{n}, \psi^{(1)}_{n} \rangle +
         \langle \psi^{(0)}_{n}, G^{(1)}\psi^{(1)}_{n} \rangle +\langle \psi^{(1)}_{n}, G^{(1)}\psi^{(0)}_{n} \rangle \right.\\
         &\left. + \langle \psi^{(0)}_{n}, G^{(2)}\psi^{(0)}_{n} \rangle  \right)
\end{split}
\end{equation}

\noindent The usual formulas are recovered by setting $G^{(1)}=G^{(2)}=0$. Following the usual procedure, one projects the second line (first order) of Equation \eqref{PerturbedEigenvalueProblem} onto $\psi_{n}^{(0)}$ with respect to the unperturbed inner product to obtain a formula for $\lambda_{n}^{(1)}$. 

\begin{equation} \label{FirstOrderLambda}
\lambda_{n}^{(1)} = \langle \psi_{n}^{(0)}, H^{(1)} \psi_{n}^{(0)} \rangle
\end{equation}

\noindent Similarly, projecting the same equation onto $\psi_{i}^{(0)}$ for $i \neq n$, together with Equation \eqref{NormalizationCondition} yields an expression for the first order correction to the eigenvector

\begin{equation} \label{FirstOrderPsi}
\psi_{n}^{(1)} = \sum_{i \neq n} \frac{\langle \psi_{i}^{(0)}, H^{(1)} \psi_{n}^{(0)} \rangle}{\lambda_{n}^{(0)}-\lambda_{i}^{(0)}} \psi_{i}^{(0)} - \frac{1}{2}\langle \psi_{n}^{(0)}, G^{(1)} \psi_{n}^{(0)} \rangle \psi_{n}^{(0)}
\end{equation}

\noindent Finally, projecting the third line (second order) of Equation \eqref{PerturbedEigenvalueProblem} onto $\psi_{n}^{(0)}$ and using the equations \eqref{NormalizationCondition}, \eqref{FirstOrderLambda} and the
above formula for $\psi_{n}^{(1)}$, one obtains the second order correction to the eigenvalues. 

\begin{equation} \label{SecondOrderLambda}
\lambda_{n}^{(2)} = \sum_{i \neq n} \frac{\langle \psi_{i}^{(0)}, H^{(1)} \psi_{n}^{(0)} \rangle \langle \psi_{n}^{(0)}, H^{(1)} \psi_{i}^{(0)} \rangle}{\lambda_{n}^{(0)} - \lambda_{i}^{(0)}} +\langle \psi_{n}^{(0)}, H^{(2)} \psi_{n}^{(0)} \rangle
\end{equation}

\noindent Notice that all dependence upon $G^{(1)}$ and $G^{(2)}$ drops out. As mentioned above, this is not surprising, as the eigenvalues of an operator do not depend upon the inner product. Compare this new expression to the usual perturbation formula given by Equation \eqref{SecondOrderQMPetrurbation}. The key difference is that $H^{(1)}$ can fail to be symmetric with respect to the unperturbed inner product. This has two important consequences. First, $\langle \psi_{i}^{(0)}, H^{(1)} \psi_{n}^{(0)} \rangle \langle \psi_{n}^{(0)}, H^{(1)} \psi_{i}^{(0)} \rangle$ is not guaranteed to be nonnegative, unlike $| \langle \psi_{i}^{(0)}, H^{(1)} \psi_{n}^{(0)} \rangle |^{2}$. This means that the main argument of Section \ref{BabyExample} cannot be directly applied in this situation. Second, in the case of a finitely degenerate unperturbed spectrum, it is no longer guaranteed that one can choose the eigenbasis of $H^{(0)}$ in a way that preserves the form of equation \eqref{SecondOrderLambda}.
Neither of those issues arise when $H^{(1)}=Q^{(1)}H^{(0)}$ and $H^{(2)}=Q^{(2)}H^{(0)}$ for some symmetric operators $Q^{(1)}$ and $Q^{(2)}$, which is the case for the perturbations studied below.

\section{Isospectral Families of Metrics on Surfaces} \label{IsospectralFamilies}

Consider an oriented compact connected Riemannian manifold $(\mathcal{M},g)$ of dimension $2$ without boundary. In this section we study analytic families of conformally equivalent isospectral metrics on such manifolds. Recall that two metrics $g_{1}$ and $g_{2}$ are conformally equivalent if there exists an $f \in C^{\infty}(\mathcal{M})$ such that $g_{1}=fg_{2}$. For a proof of the existence of isospectral non-isometric conformally equivalent metrics, see \cite{brooks1989isospectral}.

Let $g^{(0)}$ denote the unperturbed metric and let $\Delta^{(0)}$ denote the induced unperturbed Laplacian. Let $\bar{g}^{(0)}$ denote the induced metric on forms (``upper indices", or inverse metric). For now, we will consider perturbations of the inverse metric $\bar{g}^{(0)}$ rather than of the metric $g^{(0)}$, as this simplifies the computations. Let $\bar{g} = \bar{g}^{(0)} \left( 1 + \sum_{i=1}^{\infty} t^{i} f^{(i)} \right)$ be a conformal perturbation of the inverse metric, where the $f^{(i)}$ are smooth functions. In dimension $2$, it is well-known that this induces the following perturbation of the Laplace-Beltrami operator:

\begin{equation}
\Delta = \left( 1 + \sum_{i=1}^{\infty} t^{i} f^{(i)} \right) \Delta^{(0)} .
\end{equation}

\noindent Denoting the $i^{th}$ order correction to the Laplace-Beltrami operator by $\Delta^{(i)}$, the first two corrections to the Laplace-Beltrami operator are

\begin{equation} \label{ConformalLaplacian2d}
\begin{split}
    \Delta^{(1)} \psi &= f^{(1)} \left( \Delta^{(0)} \psi \right) \\
        \Delta^{(2)} \psi &= f^{(2)} \left( \Delta^{(0)} \psi \right).
\end{split}
\end{equation}

\noindent Since the corresponding perturbation series for the eigenfunctions and eigenvalues are known to converge (see \cite{bando1983generic}), one can use Equations \eqref{FirstOrderLambda} and \eqref{SecondOrderLambda}, together with the special choice of eigenbasis described in Section \ref{BabyExample}, to obtain the following formulas for the first two eigenvalue corrections:

\begin{equation}
\begin{split}
    \lambda^{(1)}_{n} &= \lambda_{n}^{(0)} \langle \psi_{n}^{(0)} , f^{(1)}\psi_{n}^{(0)} \rangle\\
    \lambda_{n}^{(2)} &= \sum_{\substack{i\\ \lambda^{(0)}_{i} \neq \lambda_{n}^{(0)} }} \frac{\lambda_{i}^{(0)}\lambda_{n}^{(0)} \left| \langle \psi_{i}^{0}, f^{(1)} \psi_{n}^{(0)} \rangle \right|^{2} }{\lambda_{n}^{(0)} - \lambda_{i}^{(0)}} + \lambda_{n}^{(0)} \langle \psi_{n}^{(0)}, f^{(2)} \psi_{n}^{(0)} \rangle .
\end{split}
\end{equation}

\noindent Notice that for a perturbation such that $f^{(2)}=0$ the above takes the desired form for our inductive argument to hold for all eigenspaces except for the one with eigenvalue zero. Indeed, for all $i,j \geq 1$, the inductive argument of Section \ref{BabyExample} yields $\langle \psi_{i}^{(0)}, f^{(1)} \psi_{j}^{(0)} \rangle=0$. Only the case of either $i$ or $j$ being zero remains. Notice that, since the zero eigenspace is that of constant functions, $\langle \psi_{i}^{(0)}, f^{(1)} \psi_{j}^{(0)} \rangle=0 ~\forall i,j \geq 1$ implies that $f^{(1)}\psi_{n}^{(0)}$ is constant for all $n \geq 1$. Pick $n$ such that $\psi_{n}^{(0)}$ vanishes at some $p \in \mathcal{M}$. Such eigenfunctions exist since eigenfunctions are smooth and functions must change sign in order to be orthogonal to the constant eigenfunction corresponding to $\lambda_{0}^{(0)}=0$. Then, $f^{(1)}(p)\psi_{n}^{(0)} (p) =0$, which implies that $f^{(1)}\psi_{n}^{(0)}=0$, since it is constant. Yet, $\psi_{n}^{(0)} \neq 0$ , thus $f^{(1)}=0$, as desired.

Notice that the above reasoning also holds for the case where $f^{(2)} \leq 0$ with isospectrality to the first two orders implying $f^{(1)}=0$ and $\langle \psi_{n}^{(0)}, f^{(2)} \psi_{n}^{(0)} \rangle=0$ for all $n$. The above can be compactly stated as follows.

\begin{lemma} \label{BasicGeomResult}
Let $(\mathcal{M},g)$ be a compact, connected, boundaryless, oriented Riemannian manifold of dimension $2$. Let $\bar{g} = \bar{g}^{(0)} \left( 1 + \sum_{i=1}^{\infty} t^{i} f^{(i)} \right)$ be a conformal perturbation of the inverse metric with $f^{(2)} \leq 0$. If the perturbation is isospectral to the first two orders, then $f^{(1)}=0$ and $\langle \psi_{n}^{(0)}, f^{(2)} \psi_{n}^{(0)}\rangle=0$ for all $n$.
\end{lemma}

This result has an interesting consequence for the structure of sets of conformally equivalent isospectral metrics in dimension $2$.

\begin{theorem} \label{InverseConcave}
Let $\bar{\mathcal{G}}$ be a set of conformally equivalent inverse metrics over a surface $\mathcal{M}$ assumed connected, oriented, compact and without boundary. Suppose that the induced Laplace-Beltrami operators are isospectral. Then, $\bar{\mathcal{G}}$ contains no convex subset composed of more than one element.
\end{theorem}
\begin{proof}
Assume $\bar{g}_{1}$ and $\bar{g}_{2}$ are distinct elements of a nonempty convex subset of $\bar{\mathcal{G}}$. Then, the line $\tau \bar{g}_{1} + (1-\tau) \bar{g}_{2}$ for $\tau \in [0,1]$ must be contained in $\bar{\mathcal{G}}$. The path $\tau \bar{g}_{1} + (1-\tau) \bar{g}_{2}$ can be studied perturbatively around any $\tau_{0} \in [0,1]$. As shown in \cite{bando1983generic} there will always exist some $\varepsilon \geq 0$ such that the perturbations series converge when $|\tau-\tau_{0}| < \varepsilon$. Notice that this corresponds to a strictly first order isospectral perturbation, i.e. $f^{(i)}=0$ for $i > 1$ and $\lambda^{(i)}=0$ for all $i$. By Lemma \ref{BasicGeomResult}, $f^{(1)} =0$, resulting in $\bar{g}_{1} = \bar{g}_{2}$, a contradiction. Thus, $\bar{\mathcal{G}}$ contains no convex subsets of more than one element. 
\end{proof}

An analogous result holds for isospectral sets in the space of metrics rather than inverse metrics. 

\begin{theorem} \label{DirectConcave}
Let $\mathcal{G}$ be a set of conformally equivalent metrics over a surface $\mathcal{M}$ assumed connected, oriented, compact and without boundary. Suppose that the induced Laplace-Beltrami operators are isospectral. Then, $\mathcal{G}$ contains no convex subset composed of more than one element.
\end{theorem}

\begin{proof}
Consider a perturbation of the metric of the form $g= g^{(0)}(1+tf^{(1)})$. For small $t$, this induces a perturbation of the inverse metric of the form $\bar{g} = \bar{g}^{(0)}(1-tf^{(1)}+t^{2} (f^{(1)})^{2}+...)$.
 Then, by assumption of isospectrality, the first order eigenvalue correction yields $\langle \psi_{n}^{(0)},f^{(1)} \psi_{n}^{(0)}  \rangle =0$. Together with the choice of eigenbasis described in Section \ref{BabyExample}, this yields 

\begin{equation}
f^{(1)} \psi_{n}^{(0)} = \sum_{\substack{i\\ \lambda^{(0)}_{i} \neq \lambda_{n}^{(0)} }} \langle \psi_{n}^{(0)}, f^{(1)} \psi_{i}^{(0)} \rangle \psi_{i}^{(0)} .
\end{equation}

\noindent Consequently,

\begin{equation}
\langle \psi_{n}^{(0)}, (f^{(1)})^{2} \psi_{n}^{(0)} \rangle = \sum_{ \substack{i\\ \lambda^{(0)}_{i} \neq \lambda_{n}^{(0)} } } \left| \langle \psi_{i}^{(0)}, f^{(1)}\psi_{n}^{(0)} \rangle \right|^{2}
\end{equation}

\noindent This allows us to express the second order correction to the eigenvalues as

\begin{equation}
\begin{split}
    \lambda_{n}^{(2)} &= \sum_{ \substack{i\\ \lambda^{(0)}_{i} \neq \lambda_{n}^{(0)} } } \left( \frac{\lambda_{i}^{(0)}\lambda_{n}^{(0)} \left| \langle \psi_{i}^{0}, f^{(1)} \psi_{n}^{(0)} \rangle \right|^{2} }{\lambda_{n}^{(0)} - \lambda_{i}^{(0)}} + \lambda_{n}^{(0)} \left| \langle \psi_{i}^{(0)}, f^{(1)} \psi_{n}^{(0)} \rangle \right|^{2} \right) \\
    &= \sum_{\substack{i\\ \lambda^{(0)}_{i} \neq \lambda_{n}^{(0)} }} \frac{(\lambda_{n}^{(0)})^{2} \left|\langle \psi_{i}^{(0)} , f^{(1)} \psi_{n}^{(0)}\rangle \right|^{2}}{\lambda_{n}^{(0)}-\lambda_{i}^{(0)}}
\end{split}
\end{equation}

\noindent Once again, we obtain an expression suitable for the inductive argument of Section \ref{BabyExample}. Thus, $f^{(1)} =0$. This rules out isospectral straight lines in the space of metrics. The remainder of the proof is identical to that of Theorem \ref{InverseConcave}.
\end{proof}

It is interesting to note that this argument does not depend on the condition $f^{(2)} \leq 0$ used in Lemma \ref{BasicGeomResult} and thus in the proof of Theorem \ref{InverseConcave}.  Indeed, in this case, $f^{(2)} = (f^{(1)})^{2} \geq 0$. Still, the expressions for the eigenvalue corrections conspire in a surprisingly nice way to yield the desired result.

Finally, note that the above theorems consider distinct isometric metrics as different objects. Thus, the isospectral sets in the above results can, if one so chooses, include metrics equivalent by isometry. Consequently, sets of conformally equivalent isometric metrics do not contain convex subsets.

\section{Outlook}

We have shown that sets of conformally equivalent isospectral metrics on surfaces contain no convex subsets. In other words, we ruled out the existence of any isospectral paths of the form $(1+tf)g_{0}$ with $f \in C^{\infty}(\mathcal{M})$.
This is an indication of the rarity of metrics that admit nontrivial isospectral deformations.
Note that, in dimension $d=2$, manifolds of negative curvature cannot be isospectrally deformed \cite{guillemin1980some}.
In fact, to the best of our knowledge, there are no known continuous isospectral families in dimension $d=2$.
Our result can thus be seen as evidence to conjecture that no such continuous families exist at all.

Beyond continuous families, our original motivation was to take a step towards a proof of the rarity of any isospectral non-isometric manifolds.
We will now discuss how our result might provide such a step. 
It is expected that isospectral non-isometric metrics are rare in the sense that they form a meagre set in some topology. See \cite{zelditch2009inverse,hezari2010inverse,zelditch1998revolution} for specific settings in which non-isometric isospectral geometries are rare in that sense. Recall that meagre sets are countable unions of nowhere dense sets and that nowhere dense sets are sets whose closure has empty interior.
Isospectral sets are known to be closed in appropriate topologies \cite{osgood1988compact} and our non-convexity result implies that isospectral sets of conformally equivalent metrics on surfaces have no interior. Consequently, our result implies that isospectral sets are nowhere dense.
If it were the case that there are only countably many isospectral sets containing non-isometric metrics, in other words that there are only countably many families of counterexamples to the reconstruction of shape from spectrum, one would thus establish that counterexamples form a meagre set.
One could then safely neglect the existence of counterexamples in most applications.
This, however, does not hold. Indeed, given a pair of isospectral non-isometric metrics, one can multiply both by the same positive number to obtain another pair of isospectral non-isometric metrics.
The spectrum of the new pair will be different as long as the factor is not equal to $1$.
In this way, uncountably many counterexamples can be constructed.
Nonetheless, it is expected that the counterexamples form a meagre set.
Consequently, we conjecture that a proof of the meagreness of non-isometric isospectral metrics would require that isospectral sets possess a property stronger than being nowhere dense.

Notice that the non-convexity result proven here is such a property.
In order to prove the nowhere density of an isospectral set $\mathcal{G}$ it is sufficient to show that at each $g_{0} \in \mathcal{G}$ there exists a continuous path of metrics $g(t)$ such that $g(0)=g_{0}$ and that lies outside the isospectral set for $t \neq 0$.
This can be easily achieved by setting $g(t)=(1+t)g_{0}$.
This family is guaranteed to exit the isospectral set for arbitrarily small $|t|$ since this perturbation of the metric modifies the volume of manifold, a quantity well-known to be determined by the spectrum.
Our non-convexity result is stronger in that it rules out all isospectral paths of the form $(1+tf)g_{0}$ with $f \in C^{\infty}(\mathcal{M})$. It thus goes beyond the minimal requirements to prove that an isospectral set is nowhere dense.
Whether this is enough to establish the meagreness of isospectral non-isometric metrics is the subject of ongoing research.

A natural further direction to pursue is to attempt to generalize the non-convexity results to non-conformal peturbations and metrics on manifolds of higher dimensions. A degrees of freedom counting argument suggests that this is likely to require the use of spectra of Laplacians on fields other than scalar fields, such as Hodge Laplacians on $p-$forms or Laplacians on covariant symmetric $2-$tensors \cite{kempf2010spacetime}.
A more involved challenge will be to generalize results of this type to metrics of indefinite signature, especially the Lorentzian metrics used in relativity.
There, the Laplacian or, more accurately, d'Alembertian is not an elliptic operator, which will require careful regularization in order to obtain a discrete spectrum.

\section{Acknowledgements}
A.K. acknowledges support through the Discovery program of the Natural Sciences and Engineering Research Council of Canada (NSERC). M.P. acknowledges support through NSERC's PGS-D program and wishes to thank Alexandre Girouard for his useful comments. The authors also wish to thank Spiro Karigiannis whose input helped clarify the presentation of this paper.


\bibliographystyle{ws-ijmpa}
\bibliography{GenericBibliography}

\end{document}